\documentclass[12pt,a4paper]{article}
\usepackage{algorithm}
\usepackage{algpseudocode}
\usepackage[utf8]{inputenc}
\usepackage[T1]{fontenc}
\usepackage{lmodern}
\usepackage{enumitem}
\usepackage{amsmath, amssymb, amsthm}
\usepackage{authblk}
\usepackage{float}
\usepackage{array}
\usepackage{booktabs}
\usepackage{graphicx}      
\usepackage{subcaption}    
\usepackage{caption}       
\usepackage{tikz}
\usepackage[a4paper, margin=1in]{geometry}

\usepackage[hidelinks]{hyperref}

\renewcommand{\arraystretch}{1.3}
\usepackage{longtable}
\usepackage[margin=1in]{geometry}
\usepackage{sectsty}
\allsectionsfont{\normalfont\scshape}
\usepackage{algorithm}
\usepackage{algpseudocode}

\usepackage{float} 
\usepackage[T1]{fontenc}

\theoremstyle{remark} 
\newtheorem{remark}{Remark}
\theoremstyle{example} 
\newtheorem{example}{Example}
\theoremstyle{proposition}
\newtheorem{proposition}{Proposition}

\title{Extending Hridaya Kolam to Multiple Loops: A Study of Non-Coprime Dot--Arm Structures}

\author[1]{Atanu Manna\thanks{Email: atanuiitkgp86@gmail.com}}
\author[2]{Suvra Kanti Chakraborty\thanks{Email: suvrakanti@vidyamandira.ac.in (Corresponding Author)}}

\affil[1]{\small Indian Institute of Carpet Technology, Chauri Road, Bhadohi--221401, Uttar Pradesh, India}
\affil[2]{\small Ramakrishna Mission Vidyamandira, Belur Math--711202, West Bengal, India}
\date{}
\begin{document}
\maketitle
\begin{abstract}
\noindent This paper extends Hridaya Kolam patterns to cases where the number of dots ($m$) and arms ($n$) are not coprime, i.e., $\gcd(m, n) \ne 1$. Such configurations give rise to multiple disjoint closed loops. We propose a modular-arithmetic-based algorithm to systematically generate such patterns, and illustrative patterns for various non-coprime $(m, n)$ pairs are provided to demonstrate the resulting multi-loop structures. 
These multi-loop Kolam designs can inspire architectural motifs and ornamental patterns in floor plans, facades, and decorative elements.
\noindent \textbf{MSc Classification Code}: 00A66, 11A07, 68U05
\end{abstract}

\section{Introduction}

Kolam (or Kamalam) is a traditional Indian floor art of geometric and symmetrical patterns, drawn with rice flour or limestone powder, especially in South India. Known by different regional names—such as Muggulu, Alpana, Mandana, and Rangoli—these designs symbolize devotion, cultural identity, and aesthetic order. Their symbolic geometry and use of natural materials have drawn interdisciplinary interest, especially in anthropology and computational design.\\

\noindent A foundational study of Hridaya Kamalam Kolam featuring radial dot arrangements was conducted by Siromoney \cite{SIROMONEY78} and later by Siromoney and Chandrasekaran \cite{SIROCHANDRA}, focusing on cases with an odd number of dots per arm and coprime arm counts. Ascher's formal language approach \cite{ASCHER}, Naranan's Fibonacci-based constructions \cite{NARANAN, NARANAN2, NARANAN3, NARANAN4}, and the diagrammatic analyses by Yanagisawa and Nagata \cite{YANAGISAWA} have expanded the theoretical framework. Robinson’s ‘Pasting Scheme’ \cite{ROBINSON} and Nagata’s digitalization work \cite{NAGATA, NAGATA2} advanced computational modeling. Mathematical connections with graph theory and algebra have been explored by Thirumuthi and Simic-Muller \cite{THIRUMURTHY}, and Sarin \cite{SARIN}.\\

\noindent Recently in 2023, \noindent Srinivasan \cite{SRINIVASAN} has investigated the Hridaya Kolam design for an odd number of dots ($m$) plotted on various numbers of arms ($n$), under the condition that $\gcd(m, n) = 1$. However, his work does not provide a solution for even numbers of dots. To address this gap, in 2025, Chakraborty \& Manna \cite{CHAKMAN} have recently introduced a novel approach using modular arithmetic to generate Hridaya Kolam designs for even numbers of dots and studied the structural properties of these patterns. They have also proposed an application in the textile sector.\\

\noindent Thus far, all research in this direction has assumed that $\gcd(m, n) = 1$. 
However, the behavior of the Kolam patterns when $\gcd(m, n) \neq 1$ remains unexplored. 
Moreover, no prior study has addressed this case. 
This motivates us to explore possible answers to the following questions.

\medskip
\textsf{Q(a)}~Can we extend Hridaya Kolam patterns to the case when $\gcd(m, n) \neq 1$?

\medskip
If the answer is affirmative, then:

\medskip
\textsf{Q(b)}~What will be the impact on Hridaya Kolam patterns if $\gcd(m, n) \neq 1$?

\medskip
Moreover,

\medskip
\textsf{Q(c)}~Are there any additional assumptions needed, apart from $\gcd(m, n) \neq 1$?

\medskip
and finally,

\medskip
\textsf{Q(d)}~Can the modular arithmetic approach help us to generate the underlying sequences for these patterns?\\

\noindent Therefore, the main objective of this present paper is to address these questions. A preliminary investigation suggests that Hridaya Kolam patterns can indeed be extended to the case where $\gcd(m, n) \neq 1$, thus answering Q(a). We will observe that in this case, the nature of the Hridaya Kolam patterns differs significantly from those where $\gcd(m, n) = 1$, thereby answering Q(b).\\

\noindent When $\gcd(m, n) \neq 1$ and $m \mid n$, the resulting pattern degenerates into a trivial central structure without meaningful loops, and is therefore excluded from our primary discussion. Thus, we impose the condition $m \nmid n$, answering Q(c). We observe that in cases where $\gcd(m, n) \neq 1$ and $m \nmid n$, the classical Kolam pattern no longer results in a single Eulerian circuit. Instead, it yields multiple disjoint closed loops. This motivates the need to determine (i) the number of disjoint loops, (ii) a suitable generating sequence, and (iii) an algorithm to construct the complete pattern. We find that by modifying the modular arithmetic structure used in the $\gcd(m, n) = 1$ case (see also \cite{CHAKMAN}), it is possible to generate sequences for Hridaya Kolam patterns where $\gcd(m, n) \neq 1$, thereby addressing Q(d).

\section{Sequences Generated by Modular Arithmetic}

\noindent
In the extended framework of Hridaya Kolam patterns, certain parameter choices lead to the emergence of multiple disjoint paths. Specifically, when integers \( m \) and \( n \) satisfy \( \gcd(m,n) = d > 1 \) and \( m \nmid n \), the resulting diagram naturally fragments into \( d \) independent closed loops instead of forming a single continuous circuit. To describe and construct these loops systematically, we employ modular arithmetic sequences indexed by a parameter \( r \). Each such sequence corresponds to one of the individual cycles contributing to the overall Kolam.
\medskip

\noindent
Extending the modular techniques introduced by Chakraborty and Manna \cite{CHAKMAN}, where configurations with \( m > n \) were also considered—we outline a more general method for generating these multi-loop patterns. The key idea is to define arithmetic sequences that cycle through values modulo \( m \), with careful indexing to distinguish the separate loops.

\subsection{Defining generating sequences}

\noindent
Let \( m, n \in \mathbb{N} \) be such that \( \gcd(m,n) = d > 1 \) and \( m \nmid n \). For each residue class \( r \in \{0, 1, \dotsc, d-1\} \), we define the sequence $\{a_k^{(r)}\}$ as below:
\[
a_k^{(r)} \equiv (kn + r) \mod m, \quad \text{where } k \in \mathbb{N}_0.
\]
To ensure values of $\{a_k^{(r)}\}$ lie in \( \{1, 2, \dotsc, m\} \), we interpret `\( 0 \mod m \)' as \( m \), so that
\[
a_k^{(r)} := \begin{cases}
(kn + r) \mod m & \text{if } (kn + r) \not\equiv 0 \mod m, \\
m & \text{if } (kn + r) \equiv 0 \mod m.
\end{cases}
\]

\noindent
Let us define
\[
S^{(r)} = [a_0^{(r)}, a_1^{(r)}, \dotsc, a_t^{(r)}],
\]
where $a_{t+1}^{(r)} = a_0^{(r)}$ and each sequence \( S^{(r)}\) is terminated at the point where it cycles back to its initial value. The \( d \) number of sequences \( \{S^{(0)}, S^{(1)}, \dotsc, S^{(d-1)}\} \) collectively cover all \( m \) vertical dot levels in the pattern. In the following we have illustrated this idea by giving some examples.

\begin{example}
  Suppose that $(m,n) = (4,6)$. Then we have $\gcd(4,6) = 2$, which implies that $r \in \{0,1\}$. Hence by using the formula: $a_k^{(r)} = (kn + r) \mod m$, one obtains the following. Choose $k=0, 1, 2, 3$ then for $r=0$, we have
    \[
    \begin{aligned}
    & 0 \mod 4 = 0 \Rightarrow 4,\\
    & 6 \mod 4 = 2, \\
    &12 \mod 4 = 0 \Rightarrow 4,\\
    & 18 \mod 4 = 2.
    \end{aligned}
    \]
Therefore, the desired sequence is $(4, 2, 4, 2, \ldots)$.
Similarly, for $r=1$ and $k=0, 1, 2, 3$, one gets
    \[
    \begin{aligned}
    &1 \mod 4 \Rightarrow 1,\\
    &  7 \mod 4\Rightarrow 3, \\
    & 13 \mod 4 \Rightarrow 1,\\
    & 19 \mod 4 \Rightarrow 3.
    \end{aligned}
    \]
Hence, we get another generating sequence in this case as $(1, 3, 1, 3, \ldots)$.\\
The following figure demonstrates how these two sequences produces a Hridaya Kolam with two closed loops. The blue lines drawn by using the sequence $(4, 2, 4, 2, \ldots)$ whereas the sequence $(1, 3, 1, 3, \ldots)$ is used to draw sky-blue lines.
\begin{figure}[H]
\centering
\begin{minipage}{0.45\textwidth}
    \centering
    \includegraphics[width=\textwidth]{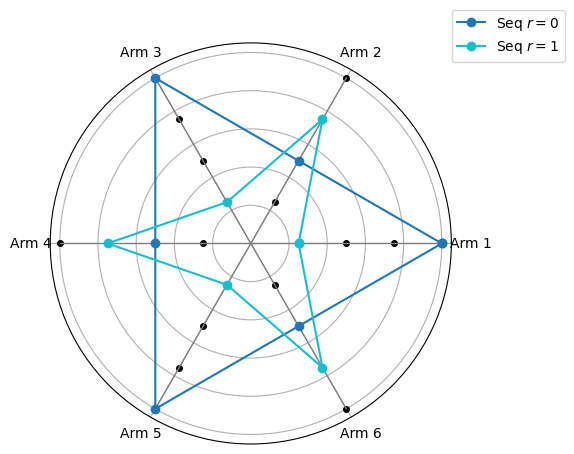}
    \caption*{Figure: \(m=4, n=6\)}
\end{minipage}
\end{figure}

\end{example}
   
\begin{example}
Now we consider $(m,n) = (6,4)$. Then again $\gcd (6, 4) = 2$ and $r\in \{0, 1\}$. Applying $a_k^{(r)} = (kn + r) \mod m$. Therefore, for $r=0$, we have the following
    \[
    \begin{aligned}
    &0 \mod 6 = 0 \Rightarrow 6,\\
    &4 \mod 6 = 4, \\
    & 8 \mod 6 = 2,\\
    & 12 \mod 6 = 0 \Rightarrow 6.
    \end{aligned}
    \]
which gives the sequence $(6, 4, 2, 6, \ldots)$. By choosing $r=1$, we get
    \[
    \begin{aligned}
    &1 \mod 6 \Rightarrow1,\\
   & 5 \mod 6 \Rightarrow 5, \\
    &9 \mod 6 \Rightarrow 3,\\
    & 13 \mod 6 \Rightarrow 1.
    \end{aligned}
    \]
which generates the sequence $(1, 5, 3, 1, \ldots)$.
\begin{figure}[H]
\centering
\begin{minipage}{0.45\textwidth}
    \centering
    \includegraphics[width=\textwidth]{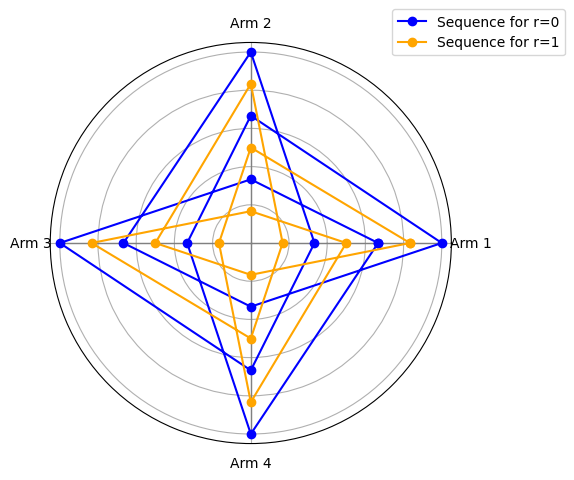}
    \caption*{Figure: \(m=6, n=4\)}
\end{minipage}
\end{figure}
The above figure demonstrates that how blue lines and orange lines form two different closed loops.
\end{example}

\subsection{Multi-Loop Kolam Structures via Modular Sequences}
\noindent
For positive integers \(m\) and \(n\), the integers modulo \(m\) can be partitioned into \(\gcd(m,n)\) distinct residue classes. The construction of multiple disjoint loops in a Kolam pattern then follows a simple rule, which is precisely given in the following proposition. Multi-loop structures arise whenever \(\gcd(m,n) > 1\).

\begin{proposition}
Let $m,n$ be positive integers and let $d=\gcd(m,n)>1$. 
For each residue $r\in\{0,1,\dots,d-1\}$ define the sequence
\[
a^{(r)}_k \equiv (k n + r) \mod m, \qquad k=0,1,2,\dots
\]
where residues modulo $m$ are represented by $\{0,1,\dots,m-1\}$. 
Then the $m$ integers $0,1,\dots,m-1$ are divided into exactly $d$ disjoint repeating sequences of this form. 
Hence, the number of disjoint loops produced by the construction equals $d=\gcd(m,n)$.
\end{proposition}

\begin{proof}
Let $d=\gcd(m,n)$. 
Then we may write $m=d\,m_1$ and $n=d\,n_1$ where $\gcd(m_1,n_1)=1$.

\noindent Consider the first sequence corresponding to $r=0$:
\[
a^{(0)}_k \equiv (k n) \mod m = (k d n_1) \mod{(d m_1)}.
\]
Two terms $a^{(0)}_k$ and $a^{(0)}_{k'}$ are equal modulo $m$ if and only if 
\[
m \mid (k-k')n \quad \iff\quad d m_1 \mid (k-k')d n_1 
\quad \iff\quad m_1 \mid (k-k')n_1.
\]
Since $\gcd(m_1,n_1)=1$, this happens exactly when $m_1\mid (k-k')$. 
Therefore the sequence repeats after every $m_1$ steps, i.e.
\[
a^{(0)}_{k+m_1} \equiv a^{(0)}_k \mod m.
\]

Hence, the sequence $a^{(0)}_k$ contains exactly $m_1 = \dfrac{m}{d}$ distinct values before repeating. 
Similarly, generalizing this for any $r$, it can be shown that the sequence $a^{(r)}_k$ also contains exactly $\dfrac{m}{d}$ distinct elements before repetition.\\

\noindent Now observe that for any $r\in\{0,1,\dots,d-1\}$,
\[
a^{(r)}_k \equiv (k n + r) \mod m = (k d n_1 + r) \mod{d m_1}.
\]
\noindent
So we can write
\[
a^{(r)}_k = k d n_1 + r + t (d m_1)
\]
for some integer \(t\).\\

\noindent
Looking at the right-hand side, we note that \(k d n_1\) is divisible by \(d \;\Rightarrow\; k d n_1 \equiv 0 \pmod{d}\) and \(t (d m_1)\) is also divisible by \(d \;\Rightarrow\; t d m_1 \equiv 0 \pmod{d}\) 

\noindent
Therefore,
\[
a^{(r)}_k \equiv r \pmod{d}.
\]

\noindent Consequently, two sequences $a^{(r_1)}_k$ and $a^{(r_2)}_k$ can never have a common term unless $r_1\equiv r_2 \mod d$. 
Since we have taken $r=0,1,\dots,d-1$, the $d$ sequences are pairwise disjoint.

\noindent Each sequence contains exactly $\dfrac{m}{d}$ distinct elements, and all $d$ sequences together cover all $m$ integers $0,1,\dots,m-1$. 
Therefore, the total number of disjoint sequences (or loops) is exactly $d=\gcd(m,n)$.
 
\end{proof}

    


\begin{remark}
    If \( m \nmid n \) and \( \gcd(m,n) > 1 \), then modular loop sequences provide the foundation for constructing multi-loop kolam patterns.
\end{remark}
\begin{remark}
    If m divides n and $gcd (m, n)>1$ then the case reduces to degeneracy of the dot-arm structure.
\end{remark}

\section{Algorithm to Generate Modular Sequences and Draw Multiple Loops}

\begin{algorithm}[H]
\caption{Visualize Modular Sequences on Radial Arms}
\begin{algorithmic}[1]
\Require Integers $m, n$ such that $\gcd(m, n) > 1$ and $m \nmid n$
\Ensure Polar plots of closed modular loops for each residue $r \in \{0, 1, \dots, \gcd(m,n)-1\}$

\State $d \gets \gcd(m, n)$
\State Define arm angles $\theta_i = \dfrac{2\pi i}{n}$ for $i = 0$ to $n-1$
\For{each $\theta_i$}
    \State Draw radial arm from origin to unit radius at angle $\theta_i$
\EndFor

\For{each residue class $r = 0$ to $d - 1$}
    \State Initialize $visited \gets \emptyset$, \quad $coords \gets [\ ]$, \quad $k \gets 0$
    \While{true}
        \State $\text{dot} \gets (k \cdot n + r) \bmod m$
        \If{$\text{dot} = 0$}
            \State $\text{dot} \gets m$
        \EndIf
        \State $arm \gets k \bmod n$
        \If{$(arm, \text{dot}) \in visited$}
            \State \textbf{break}
        \EndIf
        \State Add $(arm, \text{dot})$ to $visited$
        \State $\theta \gets \theta_{arm}$,\quad $R \gets dot$
        \State Append $(\theta,\ R)$ to $coords$
        \State $k \gets k + 1$
    \EndWhile
    \State Append starting point $coords[0]$ to $coords$
    \State Plot each curve using a distinct color and label
\EndFor
\end{algorithmic}
\end{algorithm}

\subsection{ Generated sequences for various values of $m$ and $n$ }
We now list the sequences generated by the formula posed in Section 2.1 for various $m, n$ ($2 < m, n \leq 12$) such that $\gcd(m,n) > 1$ and $m \nmid n$.
\renewcommand{\arraystretch}{1.3}
\setlength{\tabcolsep}{12pt}
\begin{longtable}{|c|c|c|l|}
\hline
$m$ & $n$ & $r$ & First Cycle of Modular Sequence \\
\hline
\endfirsthead
\hline
$m$ & $n$ & $r$ & First Cycle of Modular Sequence \\
\hline
\endhead
4 & 6 & 0; 1 & $4 \rightarrow 2 \rightarrow 4$; $1 \rightarrow 3 \rightarrow 1$ \\
 & 10 & 0; 1 & $4 \rightarrow 2 \rightarrow 4$; $1 \rightarrow 3 \rightarrow 1$\\
\midrule
6 & 3 & 0; 1; 2 & $6 \rightarrow 3 \rightarrow 6$; $1 \rightarrow 4 \rightarrow 1$; $2 \rightarrow 5 \rightarrow 2$  \\
 & 4 & 0; 1 & $6 \rightarrow 4 \rightarrow 2 \rightarrow 6$; $1 \rightarrow 5 \rightarrow 3 \rightarrow 1$\\
 & 8 & 0; 1 & $6 \rightarrow 2 \rightarrow 4 \rightarrow 6$; $1 \rightarrow 3 \rightarrow 5 \rightarrow 1$\\
 & 9 & 0; 1; 2 & $6 \rightarrow 3 \rightarrow 6$; $1 \rightarrow 4 \rightarrow 1$; $2 \rightarrow 5 \rightarrow 2$\\
 & 10 & 0; 1 & $6 \rightarrow 4 \rightarrow 2 \rightarrow 6$; $1 \rightarrow 5 \rightarrow 3 \rightarrow 1$\\
\midrule
8 & 4 & 0; 1; 2; 3 & $8 \rightarrow 4 \rightarrow 8$; $1 \rightarrow 5 \rightarrow 1$; $2\rightarrow 6 \rightarrow 2$; $3 \rightarrow 7 \rightarrow 3$ \\
 & 6 & 0; 1 & $8 \rightarrow 6 \rightarrow 4 \rightarrow 2 \rightarrow 8$; $1 \rightarrow 7 \rightarrow 5 \rightarrow 3 \rightarrow 1$\\
& 10 & 0; 1 & $8 \rightarrow 2 \rightarrow 4 \rightarrow 6 \rightarrow 8$; $1 \rightarrow 3 \rightarrow 5 \rightarrow 7 \rightarrow 1$ \\
& 12 & 0; 1; 2; 3 & $8 \rightarrow 4 \rightarrow 8$; $1 \rightarrow 5 \rightarrow 1$; $2 \rightarrow 6 \rightarrow 2$; $3 \rightarrow 7 \rightarrow 3$ \\
\midrule
9 & 3 & 0; 1; 2 & $9 \rightarrow 3 \rightarrow 6 \rightarrow 9$; $1 \rightarrow 4 \rightarrow 7 \rightarrow 1$; $2 \rightarrow 5 \rightarrow 8 \rightarrow 2$ \\
& 6 & 0; 1; 2 & $9 \rightarrow 6 \rightarrow 3 \rightarrow 9$; $1 \rightarrow 7 \rightarrow 4 \rightarrow 1$; $2 \rightarrow 8 \rightarrow 5 \rightarrow 2$ \\
& 12 & 0; 1; 2 & $9 \rightarrow 3 \rightarrow 6 \rightarrow 9$; $1 \rightarrow 4 \rightarrow 7 \rightarrow 1$; $2 \rightarrow 5 \rightarrow 8 \rightarrow 2$ \\
\midrule
10 & 4 & 0; 1 & $10 \rightarrow 4 \rightarrow 8 \rightarrow 2 \rightarrow 6 \rightarrow 10$; $1 \rightarrow 5 \rightarrow 9 \rightarrow 3 \rightarrow 7 \rightarrow 1$ \\
& 5 & 0; 1; 2; 3; 4 & $10 \rightarrow 5 \rightarrow 10$; $1 \rightarrow 6 \rightarrow 1$; $2 \rightarrow 7 \rightarrow 2$; $3 \rightarrow 8 \rightarrow 3$; $4 \rightarrow 9 \rightarrow 4$ \\
& 6 & 0; 1 & $10 \rightarrow 6 \rightarrow 2 \rightarrow 8 \rightarrow 4 \rightarrow 10$; $1 \rightarrow 7 \rightarrow 3 \rightarrow 9 \rightarrow 5 \rightarrow 1$ \\
& 8 & 0; 1 & $10 \rightarrow 8 \rightarrow 6 \rightarrow 4 \rightarrow 2 \rightarrow 10$; $1 \rightarrow 9 \rightarrow 7 \rightarrow 5 \rightarrow 3 \rightarrow 1$ \\
& 12 & 0; 1 & $10 \rightarrow 2 \rightarrow 4 \rightarrow 6 \rightarrow 8 \rightarrow 10$; $1 \rightarrow 3 \rightarrow 5 \rightarrow 7 \rightarrow 9 \rightarrow 1$ \\
\midrule
12 & 3 & 0; 1 & $12 \rightarrow 3 \rightarrow 6 \rightarrow 9 \rightarrow 12$; $1 \rightarrow 4 \rightarrow 7 \rightarrow 10 \rightarrow 1$;  \\
 &  & 2 & $2 \rightarrow 5 \rightarrow 8 \rightarrow 11 \rightarrow 2$\\
& 4 & 0; 1; 2 & $12 \rightarrow 4 \rightarrow 8 \rightarrow 12$; $1 \rightarrow 5 \rightarrow 9 \rightarrow 1$; $2 \rightarrow 6 \rightarrow 10 \rightarrow 2$;  \\
&  & 3 &  $3 \rightarrow 7 \rightarrow 11 \rightarrow 3$ \\
& 6 & 0; 1; 2; 3 & $12 \rightarrow 6 \rightarrow 12$; $1 \rightarrow 7 \rightarrow 1$; $2 \rightarrow 8 \rightarrow 2$; $3 \rightarrow 9 \rightarrow 3$; \\
&  & 4; 5 & $4 \rightarrow 10 \rightarrow 4$; $5 \rightarrow 11 \rightarrow 5$ \\
& 8 & 0; 1; & $12 \rightarrow 8 \rightarrow 4 \rightarrow 12$; $1 \rightarrow 9 \rightarrow 5 \rightarrow 1$; \\
&  & 2; 3 & $2 \rightarrow 10 \rightarrow 6 \rightarrow 2$; $3 \rightarrow 11 \rightarrow 7 \rightarrow 3$ \\
& 9 & 0; 1 & $12 \rightarrow 9 \rightarrow 6 \rightarrow 3 \rightarrow 12$; $1 \rightarrow 10 \rightarrow 7 \rightarrow 4 \rightarrow 1$; \\
 &  & 2 & $2 \rightarrow 11 \rightarrow 8 \rightarrow 5 \rightarrow 2$ \\
& 10 & 0 & $12 \rightarrow 10 \rightarrow 8 \rightarrow 6 \rightarrow 4 \rightarrow 2 \rightarrow 12$; \\
 &  & 1 & $1 \rightarrow 11 \rightarrow 9 \rightarrow 7 \rightarrow 5 \rightarrow 3 \rightarrow 1$ \\
\midrule
\end{longtable}

\noindent
The key insight is to examine how modular sequences traverse both arms and dot levels. Each step in the modular sequence \( a_k^{(r)} = (kn + r) \mod m \) moves across arms while advancing modulo \( m \). To ensure full coverage of all \( m \times n \) dot-arms, we need these sequences to form loops that:
\begin{enumerate}
    \item Touch every dot level (modulo \( m \)), and
    \item Visit every arm exactly once per loop.
\end{enumerate}
    
This leads to a precise criterion for full coverage:

\begin{proposition}
The sequences visit all \( m \times n \) dot--arm positions exactly once if and only if
\[
\gcd\!\left(\frac{m}{\gcd(m,n)},\,n\right)=1.
\]
\end{proposition}

\begin{proof}
Let \( d=\gcd(m,n) \) and \( \ell=\dfrac{m}{d} \).  
For each residue \( r=0,1,\dots,d-1 \), define the sequence of dot-arm pairs
\[
P^{(r)}=\Big\{\,(\text{arm},\text{dot})=(k\bmod n,\; (kn+r)\bmod m)\;:\;k=0,1,2,\dots\,\Big\}.
\]
Each sequence \(P^{(r)}\) repeats its dot pattern after \(\ell\) steps
and its arm pattern after \(n\) steps.  
Hence the pair \((\text{arm},\text{dot})\) repeats with period
\[
L=\operatorname{lcm}(\ell,n).
\]
Since there are \(d\) such residue classes, the total number of distinct
dot--arm positions visited by all sequences is
\[
N=d\cdot L
   =d\cdot\operatorname{lcm}\!\left(\frac{m}{d},n\right)
   =\frac{mn}{\gcd(\frac{m}{d},\,n)}.
\]
Therefore, the sequences cover all \(m\times n\) positions exactly once
if and only if \(\gcd(\frac{m}{d},n)=1\), which is the required condition.
\end{proof}

\begin{example}
    Let us consider an example for the case where \( m = 6 \) and \( n = 4 \), so \( \gcd(6,4) = 2 \). We divide the dot system into two sequences of length \( \ell = \frac{m}{d} = 3 \). We have two sequences with dots \( \{1,3,5\} \)  for one and another sequence with dots \( \{2,4,6\} \).\\
    Considering red dots and blue dots separately with given four arms, it boils down to a problem of two $3\times4$ dot-arm structures. Since \( \gcd(3,4) = 1 \), each loop visits all arms exactly once — forming a complete Eulerian cycle. Thus, we obtain exactly \( d = 2 \) disjoint closed loops.
\begin{center}
\textbf{6 Dot(Red: 1,3,5; Blue: 2,4,6)-4 arm Structure }
\end{center}
\begin{center}
\begin{tikzpicture}[scale=2, every node/.style={font=\scriptsize}]
  \foreach \i in {0,90,180,270} {
    \draw[gray!60, thick, ->] (0,0) -- (\i:1.5);
  }

  \foreach \i/\name in {0/Arm 1, 90/Arm 2, 180/Arm 3, 270/Arm 4} {
    \node at (\i:1.85) {\name};
  }

  \foreach \i in {1,...,6} {
    \pgfmathsetmacro{\r}{0.3 + 0.2*\i}
    \pgfmathtruncatemacro{\coltest}{mod(\i,2)}
    \ifnum\coltest=1
      \def\col{red}
    \else
      \def\col{blue}
    \fi

    \foreach \j/\angle in {1/0, 2/90, 3/180, 4/270} {
      \coordinate (A\j-\i) at (\angle:\r);
      \fill[\col] (A\j-\i) circle (1.5pt);
      \node[right] at (A\j-\i) {\i};
    }
  }
\end{tikzpicture}
\end{center}

\begin{minipage}[t]{0.48\textwidth}
\centering
\textbf{3 Dot(Red)-4 arm Structure}

\begin{tikzpicture}[scale=2, every node/.style={font=\scriptsize}]
  \foreach \i in {0,90,180,270} {
    \draw[gray!30, thick, ->] (0,0) -- (\i:1.5);
  }
  \foreach \i/\name in {0/Arm 1, 90/Arm 2, 180/Arm 3, 270/Arm 4} {
    \node at (\i:1.7) {\name};
  }

  \foreach \i in {1,3,5} {
    \pgfmathsetmacro{\r}{0.3 + 0.2*\i}
    \foreach \j/\angle in {1/0, 2/90, 3/180, 4/270} {
      \coordinate (R\j-\i) at (\angle:\r);
      \fill[red] (R\j-\i) circle (1.5pt);
      \node[right] at (R\j-\i) {\i};
    }
  }
\end{tikzpicture}
\end{minipage}
\hfill
\begin{minipage}[t]{0.48\textwidth}
\centering
\textbf{3 Dot(Blue)-4 arm Structure}

\begin{tikzpicture}[scale=2, every node/.style={font=\scriptsize}]
  \foreach \i in {0,90,180,270} {
    \draw[gray!30, thick, ->] (0,0) -- (\i:1.5);
  }
  \foreach \i/\name in {0/Arm 1, 90/Arm 2, 180/Arm 3, 270/Arm 4} {
    \node at (\i:1.85) {\name};
  }

  \foreach \i in {2,4,6} {
    \pgfmathsetmacro{\r}{0.3 + 0.2*\i}
    \foreach \j/\angle in {1/0, 2/90, 3/180, 4/270} {
      \coordinate (B\j-\i) at (\angle:\r);
      \fill[blue] (B\j-\i) circle (1.5pt);
      \node[right] at (B\j-\i) {\i};
    }
  }
\end{tikzpicture}
\end{minipage}

 So for \( \gcd(m,n) = d > 1 \), the dot structure naturally decomposes into \( d \) independent loops. If \( \frac{m}{d} \) and \( n \) are co-prime, each loop forms a complete Eulerian cycle over the arms. Thus, these modular sequences efficiently exhaust all dot arms in distinct and symmetric components.
\end{example}

\section{Multiple loop visualization}
\subsection{Exhaustive Coverage of Dot-Arm Structures}
In this section, we represents multiple loop kolam designs for various dots (m) and arms (n) such that 
$\gcd\left(\frac{m}{\gcd(m,n)}, n\right) = 1$. The corresponding sequences are given in the table of Section 3.1. We observe that under this condition dots-arms are exhausted for all the diagrams.

\begin{figure}[H]
\centering
\begin{subfigure}[t]{0.45\textwidth}
\includegraphics[width=\linewidth]{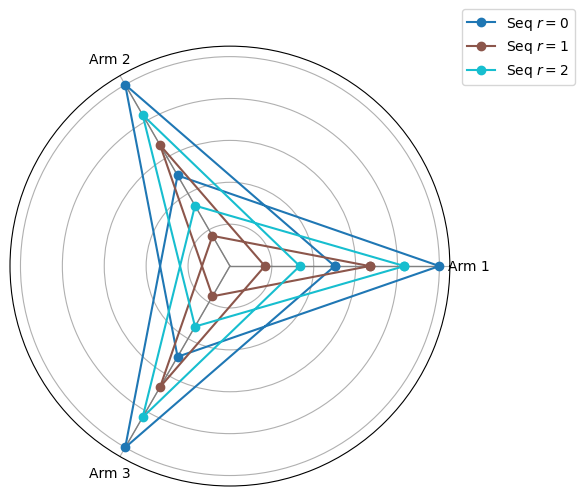}
\caption*{$(m,n) = (6,3)$}
\end{subfigure}
\hfill
\begin{subfigure}[t]{0.45\textwidth}
\includegraphics[width=\linewidth]{m6n4.png}
\caption*{$(m,n) = (6,4)$}
\end{subfigure}
\end{figure}

\begin{figure}[H]
\centering
\begin{subfigure}[t]{0.45\textwidth}
\includegraphics[width=\linewidth]{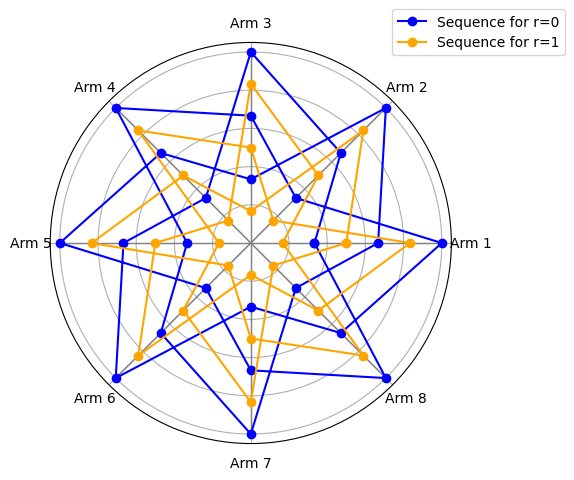}
\caption*{$(m,n) = (6,8)$}
\end{subfigure}
\hfill
\begin{subfigure}[t]{0.45\textwidth}
\includegraphics[width=\linewidth]{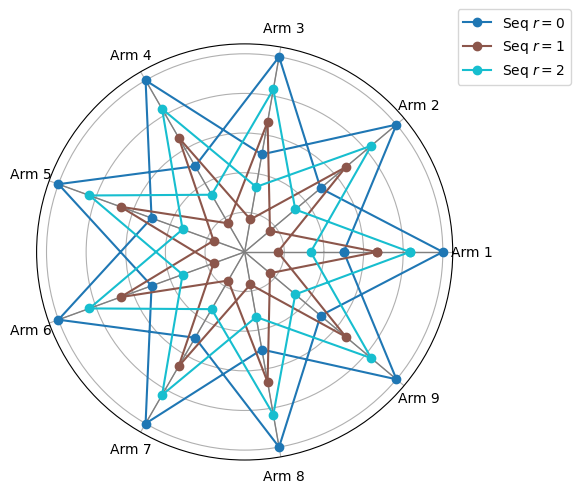}
\caption*{$(m,n) = (6,9)$}
\end{subfigure}
\end{figure}

\begin{figure}[H]
\centering
\begin{subfigure}[t]{0.45\textwidth}
\includegraphics[width=\linewidth]{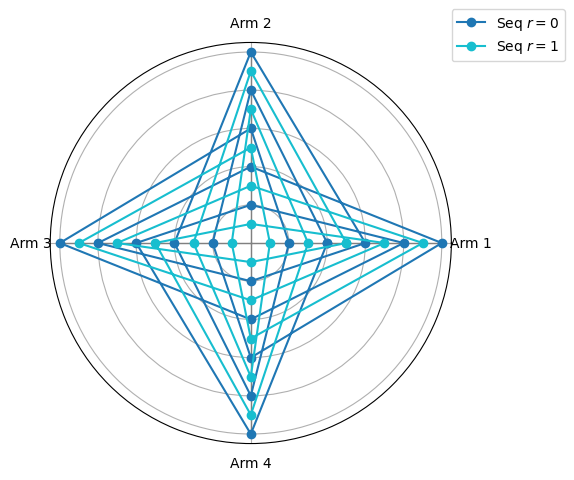}
\caption*{$(m,n) = (10,4)$}
\end{subfigure}
\hfill
\begin{subfigure}[t]{0.45\textwidth}
\includegraphics[width=\linewidth]{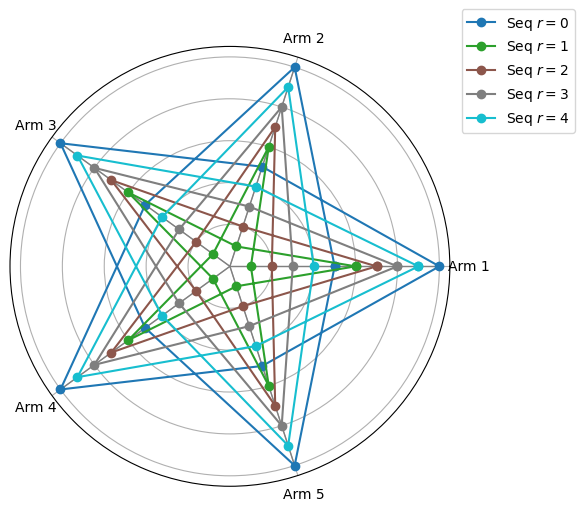}
\caption*{$(m,n) = (10,5)$}
\end{subfigure}
\end{figure}

\begin{figure}[H]
\centering
\begin{subfigure}[t]{0.45\textwidth}
\includegraphics[width=\linewidth]{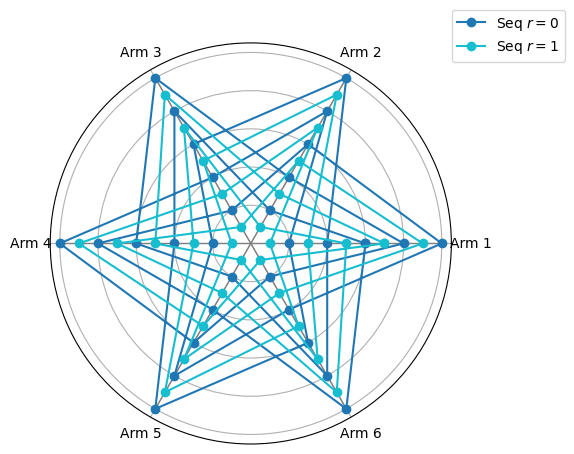}
\caption*{$(m,n) = (10,6)$}
\end{subfigure}
\hfill
\begin{subfigure}[t]{0.45\textwidth}
\includegraphics[width=\linewidth]{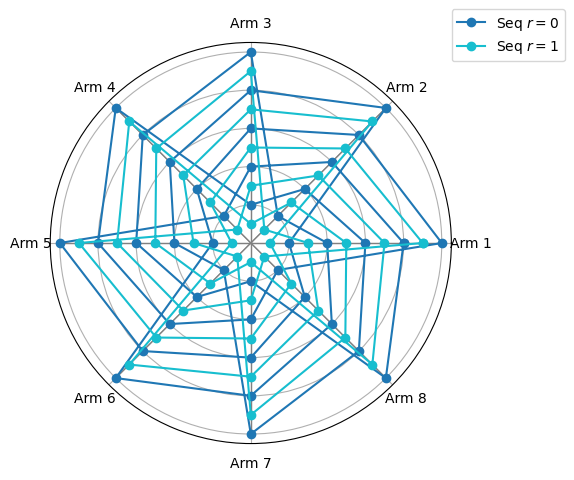}
\caption*{$(m,n) = (10,8)$}
\end{subfigure}
\end{figure}

\begin{figure}[H]
\centering
\begin{subfigure}[t]{0.45\textwidth}
\includegraphics[width=\linewidth]{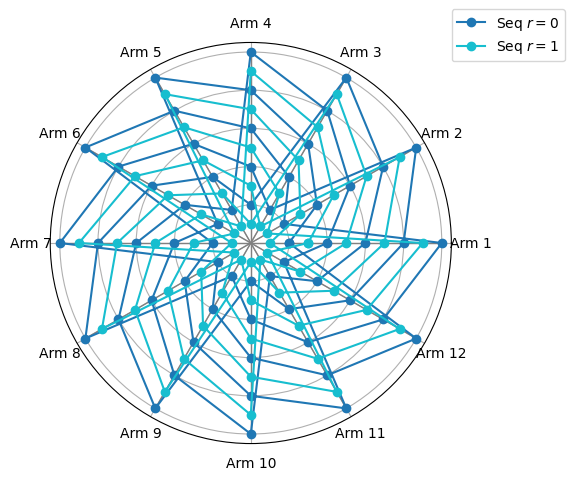}
\caption*{$(m,n) = (10,12)$}
\end{subfigure}
\hfill
\begin{subfigure}[t]{0.45\textwidth}
\includegraphics[width=\linewidth]{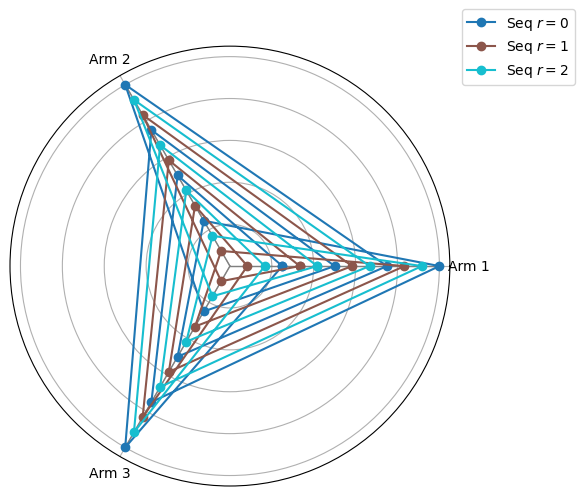}
\caption*{$(m,n) = (12,3)$}
\end{subfigure}
\end{figure}

\begin{figure}[H]
\centering
\begin{subfigure}[t]{0.45\textwidth}
\includegraphics[width=\linewidth]{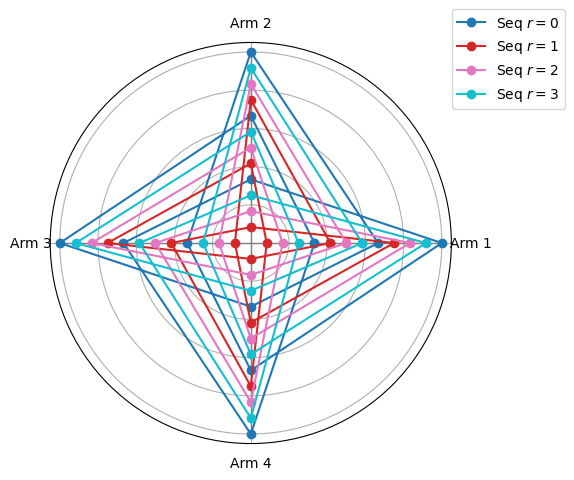}
\caption*{$(m,n) = (12,4)$}
\end{subfigure}
\hfill
\begin{subfigure}[t]{0.45\textwidth}
\includegraphics[width=\linewidth]{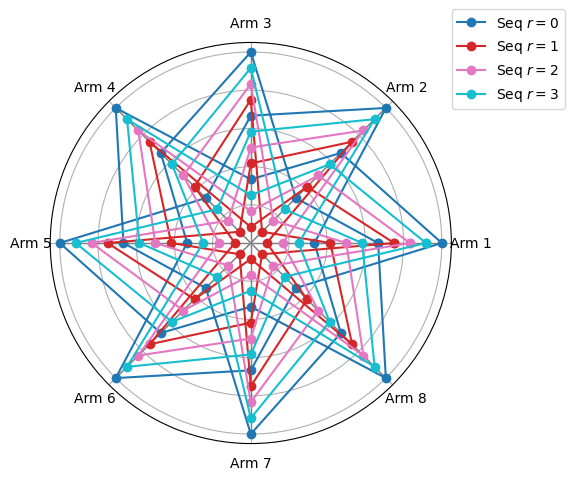}
\caption*{$(m,n) = (12,8)$}
\end{subfigure}
\end{figure}

\begin{figure}[H]
\centering
\begin{subfigure}[t]{0.45\textwidth}
\centering
\includegraphics[width=\linewidth]{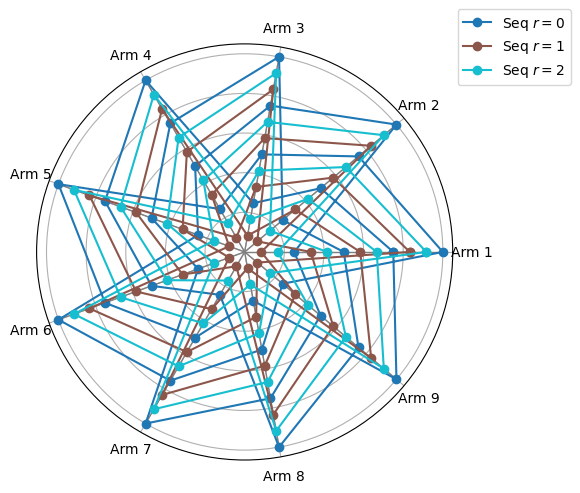}
\caption*{$(m,n) = (12,9)$}
\end{subfigure}
\end{figure}

\newpage 
\subsection{Partial (Non-Exhaustive) Coverage of Dot–Arm Structures}
Here we consider the case when $\gcd\left(\frac{m}{\gcd(m,n)}, n\right) \neq 1$ so that we obtain all dots arms which are not exhausted.


\begin{figure}[H]
\centering
\begin{subfigure}[t]{0.45\textwidth}
\includegraphics[width=\linewidth]{m4n6.png}
\caption*{$(m,n) = (4,6)$}
\end{subfigure}
\hfill
\begin{subfigure}[t]{0.45\textwidth}
\includegraphics[width=\linewidth]{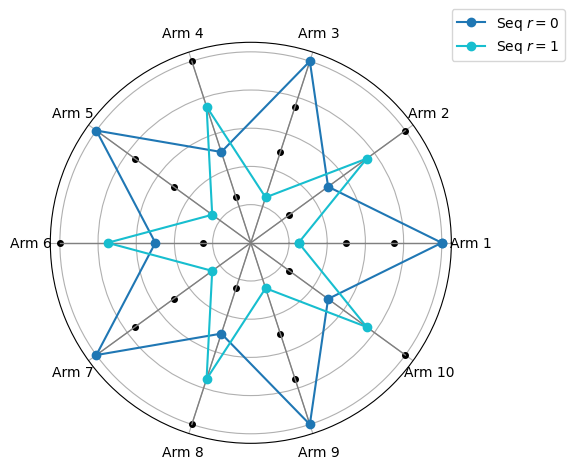}
\caption*{$(m,n) = (4,10)$}
\end{subfigure}
\end{figure}

\begin{figure}[H]
\centering
\begin{subfigure}[t]{0.45\textwidth}
\includegraphics[width=\linewidth]{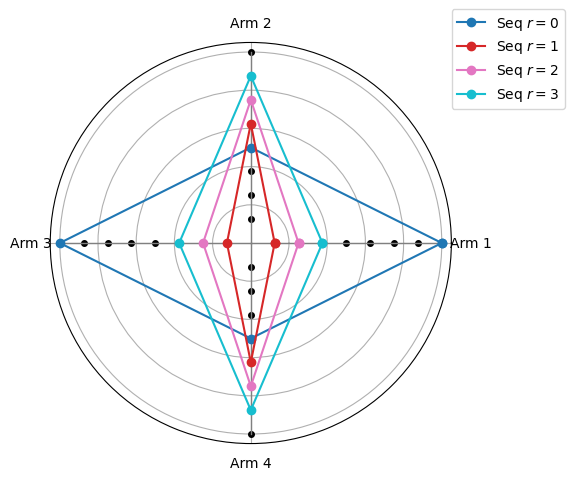}
\caption*{$(m,n) = (8,4)$}
\end{subfigure}
\hfill
\begin{subfigure}[t]{0.45\textwidth}
\includegraphics[width=\linewidth]{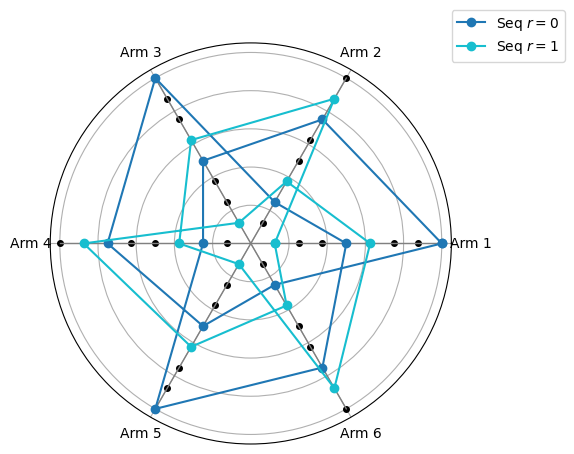}
\caption*{$(m,n) = (8,6)$}
\end{subfigure}
\end{figure}

\begin{figure}[H]
\centering
\begin{subfigure}[t]{0.45\textwidth}
\includegraphics[width=\linewidth]{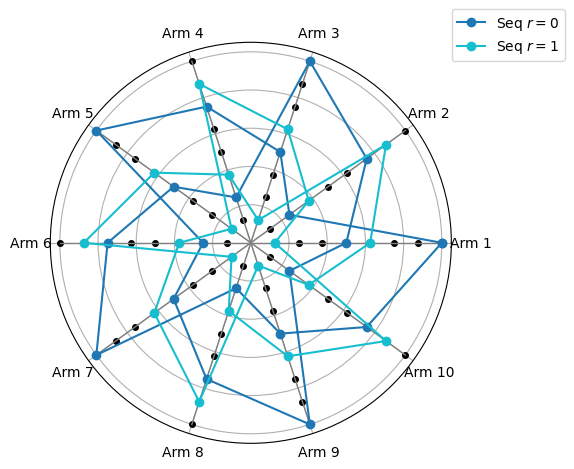}
\caption*{$(m,n) = (8,10)$}
\end{subfigure}
\hfill
\begin{subfigure}[t]{0.45\textwidth}
\includegraphics[width=\linewidth]{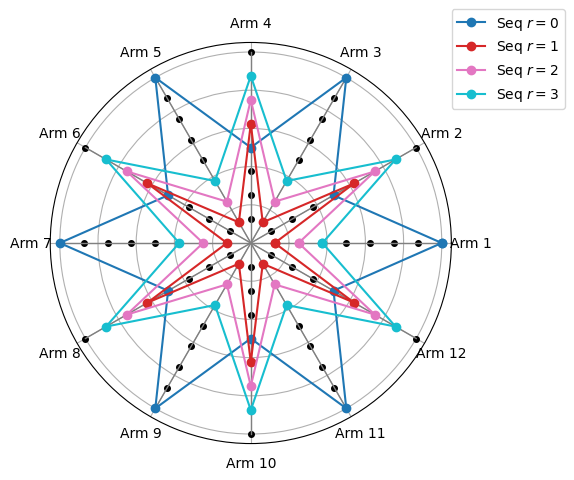}
\caption*{$(m,n) = (8,12)$}
\end{subfigure}
\end{figure}

\begin{figure}[H]
\centering
\begin{subfigure}[t]{0.45\textwidth}
\includegraphics[width=\linewidth]{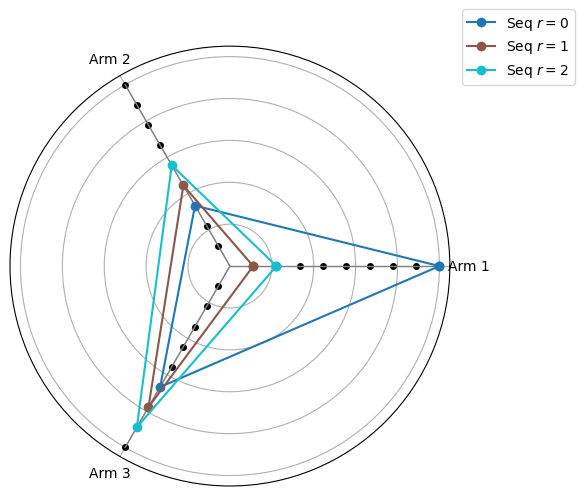}
\caption*{$(m,n) = (9,3)$}
\end{subfigure}
\hfill
\begin{subfigure}[t]{0.45\textwidth}
\includegraphics[width=\linewidth]{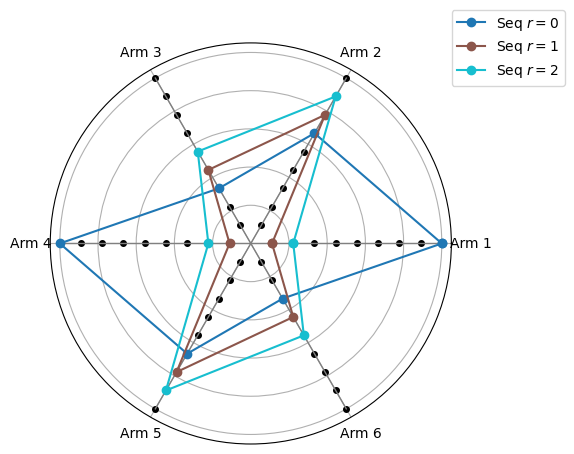}
\caption*{$(m,n) = (9,6)$}
\end{subfigure}
\end{figure}

\begin{figure}[H]
\centering
\begin{subfigure}[t]{0.45\textwidth}
\includegraphics[width=\linewidth]{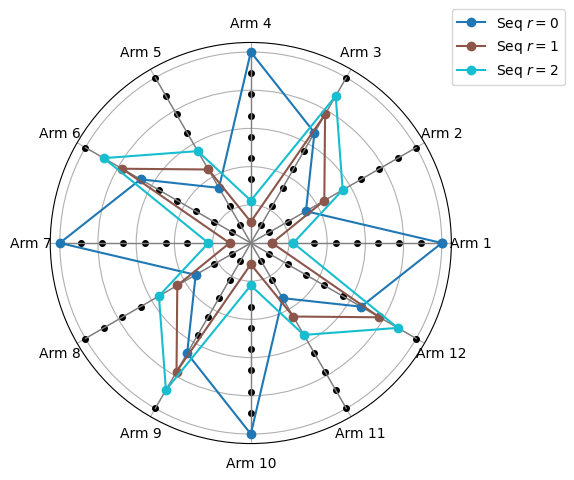}
\caption*{$(m,n) = (9,12)$}
\end{subfigure}
\hfill
\begin{subfigure}[t]{0.45\textwidth}
\includegraphics[width=\linewidth]{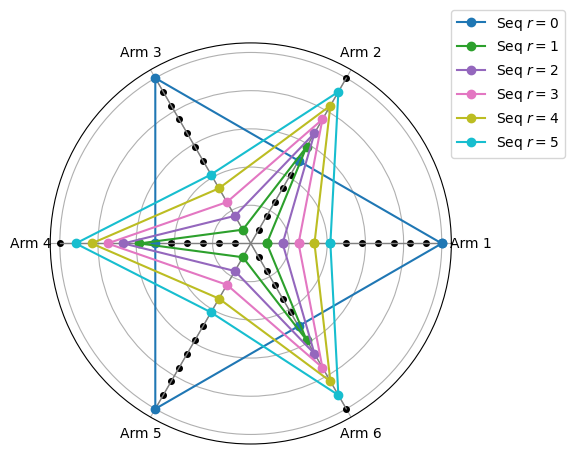}
\caption*{$(m,n) = (12,6)$}
\end{subfigure}
\end{figure}

\begin{figure}[H]
\centering
\begin{subfigure}[t]{0.45\textwidth}
\centering
\includegraphics[width=\linewidth]{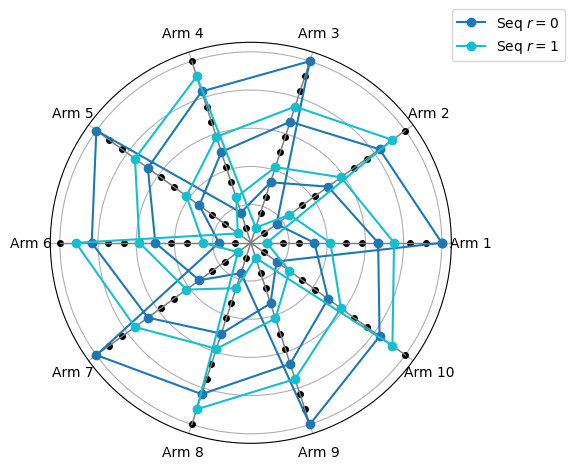}
\caption*{$(m,n) = (12,10)$}
\end{subfigure}
\end{figure}

\section{Conclusion}

\noindent
This study demonstrates how modular arithmetic sequences of the form 
\( a_k^{(r)} = (kn + r) \bmod m \) generate closed-loop structures on radial dot-arm diagrams. 
When \( \gcd(m, n) = 1 \), a single continuous loop traverses all dot-arm positions; however, 
for \( \gcd(m, n) = d > 1 \) and \( m \nmid n \), the system decomposes into \( d \) disjoint loops, 
each corresponding to a residue class \( r \in \{0, 1, \dots, d - 1\} \). 
The completeness of coverage across all dot--arm positions depends on whether 
\( \gcd\!\left(\tfrac{m}{\gcd(m, n)}, n\right) = 1 \). 
Thus, modular sequences provide a simple yet powerful framework for generating symmetric, 
aesthetically coherent Kolam patterns through purely arithmetic rules.

\end{document}